\documentclass[12pt,a4paper]{amsart}

\usepackage[latin1]{inputenc}
\usepackage{mathptmx}
\usepackage{newalg}
\usepackage{amssymb}
\usepackage{hyperref}
\usepackage{pstricks}

\def\ZZ{{\mathbb Z}}

\def\RR{{\mathbb R}}
\def\QQ{{\mathbb Q}}
\def\HH{{\mathbb H}}
\def\EE{{\mathbb E}}

\def\rank{\operatorname{rank}}
\def\cone{\operatorname{cone}}
\def\conv{\operatorname{conv}}
\def\aff{\operatorname{aff}}
\def\gp{\operatorname{gp}}
\def\hht{\operatorname{ht}}
\def\Ker{\operatorname{Ker}}
\def\depth{\operatorname{depth}}
\def\Im{\operatorname{Im}}

\def\cE{\mathcal E}

\let\dirsum=\oplus
\let\iso=\cong
\let\phi=\varphi

\newtheorem{theorem}{Theorem}[section]
\newtheorem{lemma}[theorem]{Lemma}
\newtheorem{corollary}[theorem]{Corollary}

\theoremstyle{definition}
\newtheorem{remark}[theorem]{Remark}

\newtheorem{example}[theorem]{Example}

\numberwithin{equation}{section}

\author{Winfried Bruns}
\address{Universität Osnabrück, Institut für Mathematik, 49069 Osnabrück, Germany}
\email{wbruns@uos.de}

\title{Binomial regular sequences and free sums}

\begin{document}

\begin{abstract}
Recently several authors have proved results on Ehrhart series
of free sums of rational polytopes. In this note we treat these
results from an algebraic viewpoint. Instead of attacking
combinatorial statements directly, we derive them from
structural results on affine monoids and their algebras that
allow conclusions for Hilbert and Ehrhart series. We
characterize when a binomial regular sequence generates a prime
ideal or even normality is preserved for the residue class
ring.
\end{abstract}

\maketitle

\section{Introduction}

Recently several authors have proved results on Ehrhart series
of free sums of rational polytopes; see Beck and Ho\c sten
\cite{BeH}, Braun \cite{Bra} and Beck, Jayawant, and McAllister
\cite{Be3}. In this note we treat these results from an
algebraic viewpoint. Instead of attacking combinatorial
statements directly, we derive them from structural results on
affine monoids and their algebras that allow conclusions for
Hilbert and Ehrhart series. This procedure follows the spirit
of the monograph \cite{BG} to which the reader is referred for
affine monoids and their algebras.

Our approach is best explained by the motivating example,
namely free sums of rational polytopes and their Ehrhart
series. The \emph{Ehrhart series} of a rational polytope $P$ is
the (formal) power series $E_P=\sum_{k=0}^\infty E(P,k)t^k$
where $E(P,k)$ counts the lattice points in the homothetic
multiple $kP$; see Beck and Robbins \cite{BeRo} for a gentle
introduction to the fascinating area of Ehrhart series.

One says that $R=\conv(P\cup Q)$ is the \emph{free sum} of the
rational polytopes $P$ and $Q$ if $0\in P\cap Q$, the vector
subspaces $\RR P$ and $\RR Q$ intersect only in $0$, and
$$
(\ZZ^m\cap \RR R)=(\ZZ^m\cap\RR P)+(\ZZ^m\cap\RR Q).
$$
It has been proved in \cite[Theorem 1.4]{Be3} that the Ehrhart
series of the three polytopes are related by the equation
\begin{equation*}
E_R=(1-T)E_P E_Q\eqno{(*)}
\end{equation*}
if and only if at least one of the polytopes $P$ and $Q$ is
described by inequalities of type $a_1x_1+\dots+a_mx_m\le b$
with $a_1,\dots,a_n\in \ZZ$ and $b\in\{0,1\}$.

We approach the validity of equation $(*)$ by considering the
\emph{Ehrhart monoid}
$$
\cE(P)=\{(x,k):x\in kP\cap \ZZ^m\}=\RR_+(P\times\{1\})\cap \ZZ^{m+1}.
$$
The Ehrhart series is the Hilbert series of $\cE(P)$ or,
equivalently, of the monoid algebra $K[\cE(P)]$ over a field
$K$, and therefore standard techniques for computing Hilbert
series can be applied. Ehrhart monoids are normal: if
$nx\in\cE(P)$ for some $x$ in the group $\ZZ \cE(P)$ and
$n\in\ZZ_+$, $n>0$, then $x\in\cE(P)$. The normality of a
monoid $M$ is equivalent to the normality of $K[M]$.

The free sum arises from the free join by a projection along
the line through the representatives of the origins in $P$ and
$Q$, respectively, in the free join. The algebraic counterpart
of the projection is the passage from the direct sum
$\cE(P)\dirsum \cE(Q)$ to a quotient $M$. By Corollary
\ref{dirsum}, $M$ is automatically an affine monoid in this
situation. However, the crucial question is whether $M$ is
naturally isomorphic to $\cE(R)$, and this is the case if and
only if $M$ is normal. In terms of monoid algebras, the
quotient is given by residue classes modulo a binomial.
Therefore the validity of $(*)$ can be seen as a special case
of the preservation of normality modulo a binomial in a normal
monoid algebra, for which Theorem \ref{normal} provides a
necessary and sufficient condition.

In \cite[Corollary 5.8]{Be3} the intersection of $\RR P$ and
$\RR Q$ in $0$ has been generalized to the intersection of the
affine hulls $\aff(P)$ and $\aff(Q)$ in a single rational point
$z\in P\cap Q$, and the corresponding generalization of $(*)$
follows by entirely the same argument (Corollary
\ref{rational_EE}).

Our discussion above shows that it is worthwhile to
characterize when a binomial (or more generally a regular
sequence of binomials) in an affine monoid domain generates a
prime ideal (Theorem \ref{prime} and Corollary
\ref{prime_mult}), or when even normality is preserved modulo
such a binomial (Theorem \ref{normal} and Corollary
\ref{norm_mult}). Also the main reduction step in Bruns and
Römer \cite{BR} is of this type.

It would be possible to mold the results of this note in the
language of monoids and congruences, but the ring-theoretic
environment is much richer in notions and methods, and results
like Hochster's theorem on the Cohen-Macaulay property of
normal affine monoid domains could hardly be formulated in pure
monoid theory.

This work was initiated by discussions with Serkan Ho\c{s}ten
about \cite{BeH} and then driven by the desire to prove the
results of \cite{Be3} and \cite{Bra} in an algebraic way. We
are grateful to Matthias Beck for directing our attention to
these papers, and we thank Benjamin Braun, Serkan Ho\c{s}ten,
Tyrrell McAllister and Matteo Varbaro for their careful reading
of a preliminary version and valuable suggestions.

\section{Integrality}\label{PRIME}

An \emph{affine monoid} is a finitely generated submonoid of a
group $\ZZ^m$. It is \emph{positive} if $x,-x\in M$ implies
$x=0$. For a field $K$ the monoid algebra $K[M]$ is a finitely
generated $K$-subalgebra of the Laurent polynomial ring
$K[\ZZ^m]$. We write $X^x$ for the (Laurent) monomial with
exponent vector $x$. Since the subgroup $\gp(M)$ of $\ZZ^m$
generated by $M$ is isomorphic to $\ZZ^d$ for $d=\rank M=\rank
\gp(M)$, the subalgebra $K[\gp(M)]\subset K[\ZZ^m]$ is a
Laurent polynomial ring in its own right . For an extensive
treatment of affine monoids and their algebras we refer the
reader to Bruns and Gubeladze \cite{BG}, in particular to
Chapter~4.

A \emph{(multi)grading} on a monoid $M$ is a $\ZZ$-linear map
$\deg:\gp(M)\to \ZZ^d$ for some $d>0$. If the Hilbert function
$H(M,g)=\#\{x\in M: \deg x=g\}$ is \emph{finite} for all $g$,
we can define the Hilbert series
$$
H_M(T)=\sum_{g\in \ZZ^d} H(M,g)T^g
$$
where $T$ stands for indeterminates $T_1,\dots,T_d$ and
$T^g=T_1^{g_1}\dots T_d^{g_d}$. See \cite[Ch. 6]{BG} for the
basic theorems on Hilbert series. A priori, $H_M(T)$ lives in
the $\ZZ[T_1,\dots,T_d]$-module
$\ZZ[\negthinspace[T_1,\dots,T_d]\negthinspace]$ of formal
Laurent series.

Every grading on $M$ is the specialization of the \emph{fine
grading} in which $\deg$ is simply the given embedding
$\gp(M)\hookrightarrow \ZZ^m$. We denote the Hilbert series of
the fine grading by $\HH_M$. Since $M$ can be recovered from
$\HH_M$, it is justified to call it the generating function of
$M$.

We say that $M$ is \emph{positively (multi)graded} if $\deg(M)$
is a positive submonoid of $\ZZ^d$ and the elements of
$K\subset K[M]$ are the only ones of degree $0$. This implies
the finiteness of the Hilbert function. By the classical
theorem of Hilbert-Serre, $H_M(T)$ is the Laurent series
expansion of a rational function (with respect to the positive
submonoid $\deg(M)$).

A few more pieces of terminology and notation: we say that a
nonzero $x\in \ZZ^n$ is \emph{unimodular} if $x$ generates a
direct summand. The cone generated by $A\subset \RR^n$ is
denoted by $\cone(A)$, and $\aff(A)$ is the affine subspace
spanned by $A$.

For the basic theory of zerodivisors, $R$-sequences and depth
in Noetherian rings we refer the reader to Bruns and Herzog
\cite{BH}.

\begin{theorem}\label{prime}
Let $K$ be a field, $M$ an affine monoid, and $x,y\in M$
noninvertible, $x\neq y$. Then the following statements (1) and
(2) are equivalent:
\begin{enumerate}
\item $X^x-X^y$ generates a prime ideal in $K[M]$.
\item
\begin{enumerate}
\item $X^x,X^y$ is a $K[M]$-sequence;
\item $\gp(M)/\ZZ(x-y)$ is torsionfree.
\end{enumerate}
\end{enumerate}
Moreover, if $\phi:M\to M'$ is a surjective homomorphism onto
an affine monoid $M'$ with $\rank M'=\rank M-1$ and
$\phi(x)=\phi(y)$, then (1) and (2) are equivalent to
\begin{enumerate}
\item[(3)] $K[M']=K[M]/(X^x-X^y)$ under the induced
    homomorphism.
\end{enumerate}
Finally, if in this situation $M'$ is positively multigraded
and $\phi(z)\neq 0$ for all nonzero $z\in M$, then (1), (2),
and (3) are equivalent to
\begin{enumerate}
\item[(4)] $H_{M'}=(1-T^g)H_M$ with respect to the induced
    grading on $M$, $g=\deg \phi(x)$.
\end{enumerate}
\end{theorem}

\begin{proof}
Let us start with the implication (2) $\implies$ (1). First we
prove that no monomial is a zerodivisor modulo $X^x-X^y$ if
(2)(a) holds. In fact, suppose that $X^z$ is such a
zerodivisor. Then it is contained in an associated prime ideal
$P$ of $X^x-X^y$. But $P$ is an associated prime ideal of any
nonzero element of $R=K[M]$ it contains (since $R$ is an
integral domain). Therefore $P$ is an associated prime ideal of
$X^z$ as well. Associated prime ideals of monomials are
generated by monomials \cite[4.9]{BG}, and so $P$ contains both
$X^x$ and $X^y$ together with $X^x-X^y$. This is a
contradiction since $X^x,X^y$ is a regular sequence in the
localization $R_P$.

It follows that $(X^x-X^y)$ is the contraction of its extension
to the Laurent polynomial ring $K[\gp(M)]$ (see
\cite[4.C]{BG}). So it is enough that $(X^x-X^y)K[\gp(M)]$ is a
prime ideal. This follows from (2)(b) since
$(X^x-X^y)K[\gp(M)]=(1-X^{y-x})K[\gp(M)]$ and $X^{y-x}$ is an
indeterminate in $K[\gp(M)]$ after a suitable choice of a basis
of $\gp(M)$.

For the converse we first derive (2)(a). If $(X^x-X^y)$ is a
prime ideal, then no monomial can be a zerodivisor modulo
$X^x-X^y$. On the other hand, if $X^y$ were a zerodivisor
modulo $X^x$, then it would be contained in an associated prime
ideal $P$ of $X^x$. But such $P$ is monomial and also an
associated prime ideal of $X^x-X^y$. Thus it would be equal to
$(X^x-X^y)$, which is not monomial.

(2)(b) follows from (1) since the primeness of the extension of
$(X^x-X^y)$ to $K[\gp(M)]$ evidently implies that
$\gp(M)/\ZZ(x-y)$ is torsionfree \cite[4.32]{BG}; also see
Remark \ref{facets}(c).

For the equivalence of (3) to (1) and (2) we note that the
natural surjection from $K[M]$ to $K[M']$ factors through
$K[M]/(X^x-X^y)$. Since $\rank M'=\rank M-1$, one has $\dim
K[M']=\dim K[M]-1$ (we consider Krull dimension here), and
$\Ker\phi$ is a height $1$ prime ideal. So the natural
isomorphism $K[M']=K[M]/\Ker\phi$ turns into
$K[M']=K[M]/(X^x-X^y)$ if and only if $\Ker\phi=(X^x-X^y)$.

For statement (4) to make sense, we need that $\phi(z)\neq 0$
for all $z\in M$. This assumption implies that we indeed obtain
a multigrading on $M$ by setting $\deg z=\deg \phi(z)$. The
equivalence of (4) follows by the same argument: one has
$$
H_{K[M]/(X^x-X^y)}=(1-T^g)H_{K[M]}
$$
since $X^x-X^y$ is homogeneous of degree $g$, and
$H_{K[M']}=(1-T^g)H_{K[M]}$ if and only if the two algebras are
isomorphic.
\end{proof}

\begin{remark}\label{facets}
(a) Condition (3) has been formulated in view of the
applications below. If (1) holds, then $K[M]/(X^x-X^y)$ is
automatically an affine monoid domain $K[M']$ whose underlying
monoid is the image of $M$ in $\gp(M)/\ZZ(x-y)$
\cite[4.32]{BG}.

(b) It is not hard to see that monomials
$X^{x_1},\dots,X^{x_n}$ form a $K[M]$-sequence if and only if
$X^{x_i},X^{x_j}$ is a $K[M]$-sequence for all $i\neq j$.
Nevertheless condition (2)(a) is not easy to check in general.
If $K[M]$ is Cohen-Macaulay or satisfies at least Serre's
condition $(S_2)$, for example if $M$ is normal, then (2)(a) is
equivalent to the fact that there is no facet $F$ of $\cone(M)$
with $x,y\notin F$, or, in other words, every facet contains at
least one of $x$ or $y$. Indeed, in a ring satisfying $(S_2)$
the associated prime ideals of non-zerodivisors have height
$1$, and the height $1$ monomial prime ideals are exactly those
spanned by the monomials $X^z$, $z\notin F$, for some facet $F$
of $\cone(M)$ \cite[4.D]{BG}.

(c) One should note that $K[\ZZ^d]/I$ is not only a domain, but
even a regular domain if $I$ is generated by binomials
$X^{x_1}-X^{y_1},\dots,X^{x_n}-X^{y_n}$ such that
$x_1-y_1,\dots,x_n-y_n$ generate a rank $n$ direct summand. By
induction it is enough to prove the claim for $n=1$, $x=x_1$,
$y=y_1$. With respect to a of basis of $\ZZ^d$ containing $y-x$
as the first element, $K[\ZZ^d]=K[Y_1^{\pm1},\dots,Y_d^{\pm1}]$
with $Y_1=X^{y-x}$, and $K[\ZZ^d]/(1-Y_1)$ arises from the
regular domain $K[Y_1,\dots,Y_d]/(1-Y_1)$ by the inversion of
the monomials in $Y_1,\dots,Y_d$.
\end{remark}

For a finite subset $A\subset \ZZ^m$ let the (automatically
positive) \emph{monoid $M(A)$ over $A$}  be the submonoid of
$\ZZ^{m+1}$ generated by the vectors $(x,1)\in\ZZ^{m+1}$, $x\in
A$. This type of monoid will play a special role later on, but
is useful already now for the construction of examples.

\begin{figure}[hbt]
\psset{unit=1.5cm}
$$
\def\vertex{\pscircle*{0.05}}
\begin{pspicture}(0,0)(1.6,1.6)
\pspolygon(0,0)(1,0)(0,1)
\rput(0.3,0.3){\pspolygon(0,0)(1,0)(0,1)}
\rput(0,0){\psline(0,0)(0.3,0.3)}
\rput(1,0){\psline(0,0)(0.3,0.3)}
\rput(0,1){\psline(0,0)(0.3,0.3)}
\rput(0,0){\rput(0,0){\vertex}\rput(0,1){\vertex}\rput(1,0){\vertex}}
\rput(0.3,0.3){\rput(0,0){\vertex}\rput(0,1){\vertex}\rput(1,0){\vertex}}
\psline[linestyle=dashed](0.2,1.486)(1.1,-0.186)
\rput(0.5,1.5){$x$}
\rput(1.2,0.0){$y$}
\rput(-0.2,1.0){$u$}
\rput(-0.2,0.0){$w$}
\rput(1.5,0.3){$z$}
\rput(0.15,0.45){$v$}
\end{pspicture}
\qquad
\begin{pspicture}(0,0)(2,1.6)
\pspolygon(0,0)(2,0)(2,1)(1,1)
\rput(0,0){\vertex}\rput(1,1){\vertex}\rput(1,0){\vertex}
\rput(2,0){\vertex}\rput(2,1){\vertex}
\psline[linestyle=dashed](-0.2,-0.1)(2.2,1.1)
\rput(0,-0.2){$u$}
\rput(1,-0.2){$x=y$}
\rput(2,-0.2){$z$}
\rput(2.2,0.9){$v$}
\rput(0.8,1.0){$w$}
\end{pspicture}
\qquad
\begin{pspicture}(0,0)(3,1.6)
\rput(0,0){\vertex}\rput(3,0){\vertex}\rput(1,0){\vertex}
\rput(2,0){\vertex}
\psline(0,0)(3,0)
\rput(2,-0.2){$x=y$}
\rput(1,-0.2){$u=v$}
\rput(0,-0.2){$w$}
\rput(3,-0.2){$z$}
\end{pspicture}
$$
\caption{Successive identification of lattice points}\label{succ}
\end{figure}
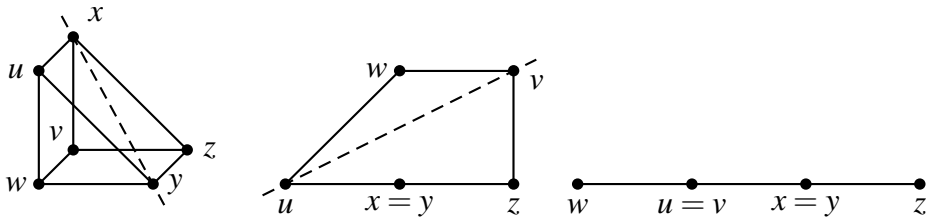
The geometry behind Theorem \ref{prime} is illustrated by
Figure \ref{succ}. We start from the monoid $M(A)$ where $A$ is
the set of vertices of the direct product of the unit
$2$-simplex and the unit $1$-simplex. The monoid $M'$ arising
from the identification of $x$ and $y$ is then defined by the
$5$ lattice points of the quadrangle in the middle, and if we
further identify $u$ and $v$, we end with the line segment on
the right with its $4$ lattice points. The polytopes in the
middle and on the right are obtained from their left neighbors
by projection along the line through the identified points,
indicated by $x=y$ and $u=v$.

We generalize the theorem to sequences of more than two
elements, leaving the generalization of (3) and (4) to the
reader.

\begin{corollary}\label{prime_mult}
With $K$ and $M$ as in Theorem \ref{prime}, let
$x_1,\dots,x_n$, $n\ge 2$, be noninvertible elements of $M$.
Then the following statements (1) and (2) are equivalent:
\begin{enumerate}
\item $X^{x_1}-X^{x_2},\dots,X^{x_{n-1}}-X^{x_n}$ is a
    $K[M]$-sequence and generates a prime ideal $P$.
\item
\begin{enumerate}
\item $X^{x_1},\dots,X^{x_n}$ is a $K[M]$-sequence;
\item $x_1-x_2,\dots,x_{n-1}-x_n$ generate a rank $n-1$
    direct summand of $\gp(M)$.
\end{enumerate}
\end{enumerate}
\end{corollary}

\begin{proof}
For the proof of the implication (1)$\implies$(2)(a) let $Q$ be
the prime ideal of $K[M]$ generated by all noninvertible
monomials. Since the associated prime ideals of monomial ideals
are themselves monomial, $X^{x_1},\dots,X^{x_n}$ is a
$K[M]$-sequence if and only if it is a $K[M]_Q$-sequence, and
the latter property follows from $\depth K[M]_{Q'}\ge n$ for
all prime ideals $Q'\supset (X^{x_1},\dots,X^{x_n})$
\cite[1.6.19]{BH}.

A prime ideal $Q'\supset (X^{x_1},\dots,X^{x_n})$ contains the
regular sequence $X^{x_1}-X^{x_2}, \dots,X^{x_{n-1}}-X^{x_n}$
of length $n-1$ that generates the prime ideal $P$. Moreover
$Q$ contains $X^{x_n}$ and $\notin P$. This implies $\depth
K[M]_{Q'}\ge n$, and (2)(a) has been verified. (2)(b) follows
since the extension of $P$ to $K[\gp(M)]$ is a prime ideal
\cite[4.32]{BG}.

For (2) $\implies$ (1) we use induction for which the starting
case $n=2$ is covered by the theorem. Let
$P'=(X^{x_1}-X^{x_2},\dots,X^{x_{n-2}}-X^{x_{n-1}})$; by
induction $K[M]/P'$ is an affine monoid domain $K[M']$ (see
Remark \ref{facets}(a)). The only critical condition is whether
$X^{x_{n-1}},X^{x_n}$ is a $K[M']$-sequence since (2)(b) of the
theorem is evidently satisfied. Let $Q'$ be a prime ideal in
$K[M']$ containing $X^{x_{n-1}},X^{x_n}$, and let $Q$ be its
preimage in $K[M]$. Then $Q$ contains the total sequence
$X^{x_1},\dots,X^{x_n}$, and we conclude $\depth K[M]_Q\ge n$.
But modulo the regular sequence
$X^{x_1}-X^{x_2},\dots,X^{x_{n-2}}-X^{x_{n-1}}$ of length $n-2$
the depth goes down by $n-2$, and therefore $\depth
K[M']_{Q'}\ge 2$. This makes it impossible that $X^n$ is a
zerodivisor modulo $X^{x_{n-1}}$ in $K[M']$.
\end{proof}

\begin{remark}\label{order}
(a) For the proof of the implication (1) $\implies$ (2)(a) of
Corollary \ref{prime_mult} we have only used that
$X^{x_1}-X^{x_2},\dots,X^{x_{n-1}}-X^{x_n},X^{x_n}$ is a
$K[M]$-sequence. The converse does also hold.

Since $(X^{x_1}-X^{x_2},\dots,X^{x_{n-1}}-X^{x_n},X^{x_n})=
(X^{x_1},\dots,X^{x_n})$, the same argument that has been used
for (1) $\implies$ (2)(a) shows that
$X^{x_1}-X^{x_2},\dots,X^{x_{n-1}}-X^{x_n},X^{x_n}$ is a
$K[M]_{Q'}$-sequence. The only problem is to lift regularity of
the sequence to $K[M]$. We can no longer use the fine grading,
but it is sufficient that there is a multigrading for which (i)
$Q'$ is the ideal generated by the noninvertible homogeneous
elements, and (ii)
$X^{x_1}-X^{x_2},\dots,X^{x_{n-1}}-X^{x_n},X^{x_n}$ are
homogeneous. Then we are dealing with homogeneous elements in
the $\vphantom{x}^*\negthinspace$maximal ideal $Q'$ of the
$\vphantom{x}^*\negthinspace$local ring $K[M]$. See
\cite[1.5.15(c)]{BH} that covers the case of positive
$\ZZ$-gradings; however, it is solely relevant that the grading
group is torsionfree (Bourbaki \cite[Ch. 4, § 3, no. 1]{Bou}).

It remains to find a suitable grading. To this end we let $U$
be the saturation of $\ZZ(x_1-x_2)+\dots+\ZZ(x_{n-1}-x_n)$ in
$\gp(M)$. Then $G=\gp(M)/U$ is torsionfree, and the natural
homomorphism $\gp(M)\to G$ is the right choice.

(b) We have assumed in Theorem \ref{prime} and in Corollary
\ref{prime_mult} that $x_1,\dots,x_n$ are noninvertible. If one
allows that one of the $x_i$ is a unit in $M$, then (2)(a)
makes no sense anymore since the definition of $K[M]$-sequence
comprises the condition $(x_1,\dots,x_n)\neq K[M]$. But
dropping this requirement and keeping only that $x_i$ is not a
zerodivisor modulo $(x_1,\dots,x_{i-1})$ for $i=1,\dots,n$ is
not the way out.

The ideal $P$ generated by the $X^{x_{i-1}}-X^{x_i}$ is
independent of the order of the $x_i$, and especially its
primeness does not depend on the order. However, the second
property in (1), namely that the generators form a
$K[M]$-sequence, may be order sensitive if one of the $x_i$ is
a unit and we have left the shelter of the ``roof'' $Q'$ above.
For a concrete example set $M=M(A)$ where $A$ is the set of the
vertices $3$-dimensional unit cube, and $x,y,z$ are chosen as
indicated in Figure \ref{cube}.
\begin{figure}[hbt]
\psset{unit=1.5cm}
$$
\def\vertex{\pscircle*{0.05}}
\begin{pspicture}(0,0)(1.6,1.6)
\pspolygon(0,0)(1,0)(1,1)(0,1)
\rput(0.3,0.3){\pspolygon(0,0)(1,0)(1,1)(0,1)}
\rput(0,0){\psline(0,0)(0.3,0.3)}
\rput(1,0){\psline(0,0)(0.3,0.3)}
\rput(0,1){\psline(0,0)(0.3,0.3)}
\rput(1,1){\psline(0,0)(0.3,0.3)}
\rput(0,0){\rput(0,0){\vertex}\rput(0,1){\vertex}\rput(1,0){\vertex}\rput(1,1){\vertex}}
\rput(0.3,0.3){\rput(0,0){\vertex}\rput(0,1){\vertex}\rput(1,0){\vertex}\rput(1,1){\vertex}}
\rput(-0.2,1.0){$x$}
\rput(-0.2,0.0){$y$}
\rput(1.2,0.0){$z$}
\end{pspicture}
$$
\caption{The unit cube}\label{cube}
\end{figure}
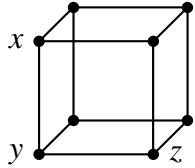
Then $X^x-X^y,X^y-X^z,X^z-1$, corresponding to $x,y,z,0\in M$,
is not a $K[M]$-sequence, although the permutation
$X^z-1,X^y-X^z,X^x-X^y$, corresponding to $0,z,y,x\in M$, is a
$K[M]$-sequence, and both sequences generate the same prime
ideal $P$. In fact, $K[M]/P$ is isomorphic to the polynomial
ring in one variable over $K$.

This not really a surprise: in a non-local situation the fact
that an ideal $P$ is generated by a regular sequence of length
$3$ does not imply that every length $3$ sequence generating
$P$ is regular.

The order that just made the generators of $P$ a
$K[M]$-sequence does always work: (2)(b) alone is equivalent to
(1), provided $x_1$ is a unit. Under this assumption all
arguments remain essentially unchanged, except that the set of
monomial ideals containing $x_1,\dots,x_n$ is automatically
empty.

(c) Binomial regular sequences in polynomial rings $K[\ZZ_+^m]$
have been investigated in Fischer, Morris and Shapiro
\cite{FMS} and Fischer and Shapiro \cite{FS}.
\end{remark}

We now turn to a situation in which the conditions of Theorem
\ref{prime} are automatically satisfied.

\begin{corollary}\label{dirsum}
Let $L$,$M$ and $N$ be affine monoids, $\phi:M\dirsum N\to L$ a
surjective homomorphism with $\rank L=\rank M+\rank N-1$, and
suppose that $\phi(x)=\phi(y)$ for $x\in M$, $y\in N$, $x\neq
0$ or $y\neq 0$. Furthermore assume that nonzero $\phi(z)\neq
0$ for all $z\in M\dirsum N$ and let $L$ be positively
multigraded such that $\deg \phi(x)$ is a unimodular element of
the grading group.

Then $K[L]\iso K[M\dirsum N]/(X^x-X^y)$ and
$H_{L}=(1-T^g)H_{M\dirsum N}$ with respect to the grading on
$M\dirsum N$ induced by the grading on $L$.
\end{corollary}

\begin{proof}
We must verify the conditions (1)(a) and (b) of the theorem.
For (a) the verification is a trivial exercise. For (b) let $G$
be the grading group of $L$. The grading on $L$ induces
gradings on $M$ and $N$ via the embedding of $M$ and $N$,
respectively, into $M\dirsum N$. Consider the homomorphism
$M\dirsum N\to G\dirsum G$, $(u,v)\mapsto (\deg u,\deg v)$.
Under this homomorphism $x-y=(x,-y)$ goes to the unimodular
element $(\deg x,-\deg y)$ of $G\dirsum G$. Therefore $x-y$ is
unimodular in $\gp(M\dirsum N)$.
\end{proof}

As a special case of Corollary \ref{dirsum} we can consider the
free sum of point configurations. Following \cite{Be3} let
$A,B\subset \RR^m$. We say that $A\cup B$ is the \emph{free
sum} of $A$ and $B$ if $0\in A\cap B$ and the vector subspaces
$\RR A$ and $\RR B$ of $\RR^m$ intersect only in $0$. The
relationship between $M(A\cup B)$ and $M(A)\dirsum M(B)$ is
given by part (1) of the next corollary in terms of monoid
algebras.

\begin{corollary}\label{freesum}
 Let $A$ and $B$ be finite subsets of $\ZZ^m$ such that $A\cup
B$ is the free sum of $A$ and $B$. Set $x=(0,1)\dirsum 0$,
$y=0\dirsum (0,1)$. Then
\begin{enumerate}
\item $K[M(A\cup B)]\iso K[M(A)\dirsum M(B)]/(X^x-X^y)$;
\item $\HH_{M(A\cup B)}=(1-T_{m+1})\HH_{M(A)}\HH_{M(B)}$.
\end{enumerate}
\end{corollary}

\begin{proof}
We set $M=M(A)$, $N=M(B)$ and $L=M(A\cup B)$. Then the natural
embeddings $M\subset L$ and $N\subset L$ induce a surjective
homomorphism $M\dirsum N\to L$, $(x,k)\dirsum(y,l)\mapsto
(x+y,k+l)$. Both $x$ and $y$ go to $(0,1)\in L\subset
\ZZ^{m+1}$, and therefore to a unimodular element in $\gp(L)$.
It only remains to apply Corollary \ref{dirsum}.
\end{proof}

\begin{remark} We have formulated Corollary \ref{freesum} for
the fine grading. Since every other grading is a specialization
of the fine grading, the formula in (2)  holds for every
coarser grading as well. In particular it holds for the
\emph{standard grading} on $M(A)$, $M(B)$ and $M(A\cup B)$ in
which $\deg(x,k)=k\in \ZZ$.

The formula in (2) was stated (for the standard grading) in
\cite[Lemma 10]{BeH} without the factor $1-T_{m+1}$. Therefore
some of the results in \cite{BeH} need an analogous correction,
but this only concerns the denominators of the Hilbert series
appearing there, and the statements about the numerator
polynomials remain untouched.
\end{remark}

The construction of free sums has been generalized in
\cite{Be3} as follows. We consider subsets $A$ and $B$ of
\begin{figure}[hbt]
\psset{unit=1.5cm}
$$
\def\vertex{\pscircle*{0.05}}
\begin{pspicture}(0,0)(1.3,1.3)
\psline(-0.3,-0.3)(1.3,1.3)
\psline(1.3,-0.3)(-0.3,1.3)
\rput(0,0){\vertex}
\rput(1,0){\vertex}
\rput(0,1){\vertex}
\rput(1,1){\vertex}
\rput(-0.2,1.0){$A$}
\rput(1.2,1.0){$B$}
\psline{->}(-0.3,0)(1.3,0)
\psline{->}(0,-0.3)(0,1.3)
\end{pspicture}
$$
\caption{Intersection in a rational point}\label{simple}
\end{figure}
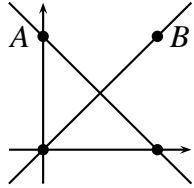
$\RR^m$ such that $\aff(A)$ and $\aff(B)$ meet in a single
point $p_0$; see Figure\ref{simple} for a very simple example.
In this example
$$
H_{M(A\cup B)}=\frac{1+T}{(1-T)^3}=(1-T^2)H_{M(A)}H_{M(B)}
$$
in the standard grading. The ``correction'' $1-T^2$ reflects
that $(2p_0,2)\in M(A)\cap M(B)$.

Let $A,B\subset \ZZ^m$ be finite and suppose that $\aff(A)$ and
$\aff(B)$ meet exactly in $p_0$; then necessarily
$p_0\in\QQ^m$. We can form $M(A)$, $M(B)$ and $M(A\cup B)$. If
$p_0\in \ZZ^m\cap A \cap B$, then we are in the situation of
Corollary \ref{freesum} after an affine-integral coordinate
transformation. But the comparison of the monoids is already
possible under a weaker assumption, as suggested by the example
above and \cite[Corollary 5.8]{Be3}, to which we will come back
in Corollary \ref{rational_EE}.

\begin{corollary}\label{rational_HH}
Let $A,B\subset \ZZ^m$ be finite and suppose that $\aff(A)$ and
$\aff(B)$ intersect in a single point $p_0$. Furthermore
suppose  $(kp_0,k)\in M(A)\cap M(B)$ for the smallest $k>0$
such that k$p_0\in \ZZ^m$. Set $x=(kp_0,k)\dirsum 0$ and
$y=0\dirsum (kp_0,k)$. Then
\begin{enumerate}
\item $K[M(A\cup B)]\iso K[M(A)\dirsum M(B)]/(X^x-X^y)$;
\item $\HH_{M(A\cup
    B)}=(1-T^{(kp_0,k)})\HH_{M(A)}\HH_{M(B)}$.
\end{enumerate}
\end{corollary}

\begin{proof}
Since $\rank M(A\cup B)=\rank M(A)+\rank M(B)-1$ under our
hypotheses, we are in the situation of Corollary \ref{dirsum},
except that the unimodularity of $x-y$ needs a different
argument: it holds since $(kp_0,k)$ has coprime entries. (Note
that we have \emph{not} defined $k$ by the condition that
$kp_0\in M(A)\cap M(B)$.)
\end{proof}

\section{Normality}

Let $P\subset \RR^m$ be a rational polytope. Then the
(ordinary) \emph{Ehrhart function} is given by
$$
E(P,k)=\# (kP\cap\ZZ^d)
$$
and the corresponding generating function $E_P\sum_{k=0}^\infty
E(P,k)T^k$ is the \emph{Ehrhart series}. In order to interpret
the Ehrhart series as a Hilbert series one forms the monoid
$$
\cE(P)=\{(x,k): x\in kP\cap\ZZ^m\}\subset \ZZ^{m+1}.
$$
By Gordan's lemma $\cE(P)$ is an affine monoid, and the Ehrhart
series of $P$ is just the standard Hilbert series of $\cE(P)$.
We define the \emph{multigraded} or \emph{fine Ehrhart series}
(or lattice point generating function) of $P$ by
$$
\EE_{P}=\HH_{\cE(P)}.
$$

It is tempting to interpret the results in Section \ref{PRIME}
as statements about Ehrhart series. Such an interpretation is
indeed possible and will be given below, but it requires
further hypotheses. Let us consider the situation of Corollary
\ref{freesum} and rational polytopes $P$ and $Q$ in $\RR^m$,
such that $0\in P\cap Q$, the vector subspaces $\RR P$ and $\RR
Q$ intersect only in $0$, and
$$
(\ZZ^m\cap \RR R)=(\ZZ^m\cap\RR P)+(\ZZ^m\cap\RR Q).
$$
Then we say that $\conv(P\cup Q)$  is the (convex) \emph{free
sum} of $P$ and $Q$ (see Henk, Richter-Gebert and Ziegler
\cite{HRZ} for further information). The free sum of polytopes
can be constructed from the free join by projecting along the
line through the representatives of the origins in the free
join. Figure \ref{join2sum} illustrates this construction.
\begin{figure}[hbt]
\psset{unit=1.5cm}
\def\vertex{\pscircle*{0.05}}
$$
\begin{pspicture}(0,0)(3,2)
\def\Va{0,0}
\def\Vb{3,0}
\def\Vc{0.9,2}
\def\Vd{2.9,1.2}
 \rput(\Va){\vertex}
 \rput(\Vb){\vertex}
 \rput(\Vc){\vertex}
 \rput(\Vd){\vertex}
 \rput(2.95,0.6){\vertex}
 \rput(0.45,1){\vertex}
 \psline[linestyle=dashed](0.3,1.024)(3.1,0.576)
 \pspolygon(\Va)(\Vc)(\Vd)(\Vb)
 \psline[linestyle=dotted](\Va)(\Vd)
 \psline(\Vc)(\Vb)
 \rput(0.1,1){$x$}
 \rput(3.3,0.6){$y$}
\end{pspicture}
\qquad
\qquad
\begin{pspicture}(0,0)(3,2)
\psline(0,1)(2,1)
\psline(1,0)(1,2)
\pspolygon(0,1)(1,0)(2,1)(1,2)
\rput(1,0){\vertex}
\rput(1,2){\vertex}
\rput(0,1){\vertex}
\rput(2,1){\vertex}
\rput(1,1){\vertex}
\rput(1.4,1.1){$x=y$}
\end{pspicture}
$$
\caption{From the free join of two line segments to their free sum}\label{join2sum}
\end{figure}
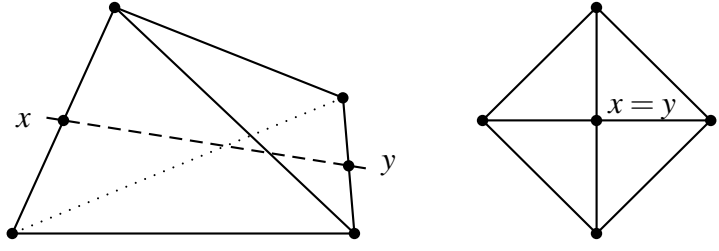

We would like to conclude that $\EE_R=(1-T_{m+1})\EE_P\EE_Q$.
This conclusion is equivalent to the fact that $\cE(R)$ arises
from $\cE(P)\dirsum\cE(Q)$ via the construction in Corollary
\ref{dirsum}. In general this is not the case, even if the
evidently necessary conditions are satisfied.

\begin{example}
Let $P\subset \RR^3$ be the lattice polytope spanned by the
points $-(e_1+e_2+e_3)$, $e_i$, $e_i+e_j$, $i,j=1,2,3$, $i\neq
j$ ($e_i$ denotes the $i$-th unit vector). For $Q$ we choose
the interval $[-1,2]\subset \RR$. Consider $P$ and $Q$ as
lattice polytopes in $\RR^4=\RR^3\dirsum \RR$. Then
$R=\conv(P\cup Q)$ is indeed the free sum of $P$ and $Q$.

Set $A=P\cap\ZZ^4$ and $B=Q\cap \ZZ^4$. That $\cE(Q)=M(B)$
holds for trivial reasons, and using Normaliz \cite{Nmz} one
checks that $\cE(P)=M(A)$. One even has $R\cap \ZZ^4 =A\cup B$.
Nevertheless $\cE(R)\neq M(A\cup B)$. This can be checked by
Normaliz directly or by inspection of the Ehrhart series:
$$
E_{P\dirsum Q}=\frac{1+3T+5T^2+4T^3+2T^4}{(1-T)^5}\neq
\frac{1+3T+4T^2+5T^3+2T^4}{(1-T)^5}
=(1-T)E_PE_Q.
$$
As we will see in Corollary \ref{freesum_normal}, this
inequality is not a surprise.
\end{example}

In the following we will have to adjoin inverse elements to the
affine monoid $M$; see \cite[p. 62]{BG}. Relative to \cite{BG}
we use the shortcut $M[-G]=M[-(G\cap M)]$ for faces $G$ of
$\cone(M)$. Extensions of type $M[-G]$ appear naturally when
localizations of monoid domains are to be considered since the
subsets $G\cap M$ of $M$ are exactly those complementary to
prime ideals. The following characterization of regular
localizations $K[M]_P$ is only implicitly  given in \cite{BG}.

\begin{lemma}\label{regular}
Let $M$ be an affine monoid of rank $d$, $K$ a field, and $P$ a
prime ideal in $K[M]$. Let $Q$ be the (automatically prime)
ideal generated by all monomials $X^x\in P$ and let $F$ be the
face of $\cone(M)$ spanned by all $y\in M$, $X^y\notin Q$. Then
the following are equivalent:
\begin{enumerate}
\item $K[M]_Q$ is a regular local ring;
\item $K[M]_P$ is a regular local ring;
\item $M[-F]$ is isomorphic to $\ZZ^{d-n}\dirsum \ZZ_+^n$
    for some $n$, $0\le n\le d$.
\end{enumerate}
\end{lemma}

\begin{proof}
 The implications (3) $\implies$ (2) $\implies$ (1) hold since
regularity is preserved under localizations.

For (1) $\implies$ (3) it is enough that $R=K[M[-F]]$ is a
regular ring; see \cite[4.45]{BG}. This follows from general
principles that hold for $\vphantom{x}^*\negthinspace$local
rings; see the discussion in \cite[p. 208]{BG}. Nevertheless a
direct argument may be welcome. The crucial observation is that
every (prime) ideal of $R$ generated by monomials is contained
in $Q'=QR$.

First we show that $R$ is normal. To this end let $\overline R$
be the normalization of $R$. It is itself an affine monoid
domain and a finitely generated $R$-module. The localization
$(\overline R/R)_{Q'}$ vanishes since $R_{Q'}$ is regular and
thus normal. But then $\overline R/R$ vanishes since its
support would have to contain a monomial prime ideal if it were
empty.

By \cite[4.45]{BG} factoriality of $R$ is sufficient for (3),
and it holds if all monomial height~$1$ prime ideals $P'$ are
principal (Chouinard's theorem \cite[4.56]{BG}). But this
follows by the same argument that shows normality: a monomial
generating the extension of $P'$ to the factorial ring $R_{Q'}$
must generate $P'$ itself.
\end{proof}

The key to results about Ehrhart series is the preservation of
normality in the situation of Theorem \ref{prime}. As we will
see, normality depends on the height of monoid elements over
facets: every $x\in M$ has a well-defined (lattice)
\emph{height} over a facet $F$ of $\cone(M)$, we denote it by
$\hht_F(x)$. It is the number of hyperplanes between $F$ and
$x$ parallel to $F$ that pass through lattice points and do not
contain $F$; so $\hht_F(x)=0$ if and only if $x\in F$.

\begin{theorem}\label{normal}
Let $M$ be a normal affine monoid of rank $d$, and Suppose that
$x,y\in M$ satisfy conditions (2)(a) and (b) of Theorem
\ref{prime}. Then the following are equivalent:
\begin{enumerate}
\item $K[M]/(X^x-X^y)$ is normal.
\item If $G$ is a subfacet of $\cone(M)$ such that
    $x,y\notin G$, then $M[-G]\iso \ZZ^{d-2}\dirsum
    \ZZ_+^2$, and $x$ or $y$ has height $1$ over one of the
    exactly two facets $F',F''$ containing $G$.
\end{enumerate}
\end{theorem}

\begin{proof}
As for Theorem \ref{prime}, we start with the implication (2)
$\implies$ (1). By Hochster's theorem $K[M]$ is Cohen-Macaulay
\cite[6.10]{BG}, and thus Theorem \ref{prime} implies that
$K[M]/(X^x-X^y)$ is Cohen-Macaulay. Moreover,
$K[M]/(X^x-X^y)\iso K[M']$ where $M'$ is the image of $M$ in
$\gp(M)/\ZZ(x-y)$.

It is enough to show that $K[M']$ satisfies Serre's condition
$(R_1)$ since $(S_2)$ follows from Cohen-Macaulayness (see
\cite[4.F]{BG} for Serre's conditions and normality). Sice
$K[M']$ is a monoid domain, it is enough to check that the
localizations with respect to monomial prime ideals of height
$1$ are regular \cite[Exerc.\ 4.16]{BG}. Let $P$ be such a
prime ideal in $K[M']$. The preimage $Q$ in $K[M]$ has height
$2$ and contains $X^x-X^y$. There are two cases to distinguish:
(i) $X^x, X^y\notin Q$ and (ii) $X^x, X^y\in Q$. In fact, $Q$
contains either both monomials or none.

Somewhat surprisingly, case (i) does not imply any other
condition on $x$ and $y$ than those occurring already in
Theorem \ref{prime}, which are satisfied by hypothesis. Let
$Q'$ be the ideal generated by all monomials in $Q$. We have
$0\neq Q'\neq Q$ since $Q'$ contains monomials, but
$X^x,X^y\notin Q$. Therefore all monomials outside the facet
$F$ of $\cone(M)$ corresponding to $Q'$ are inverted in the
passage to $K[M]_Q$. Since $M$ is normal, $M[-F]\iso
\ZZ^{d-1}\dirsum \ZZ_+$, and $x$ and $y$ belong to $\ZZ^{d-1}$
because they are not in $Q'$. Since $x-y$ is a basis element in
$\gp(M)$, it is a basis element of the subgroup $\ZZ^{d-1}$,
and $K[\ZZ^{d-1}\dirsum\ZZ_+]/(X^x-X^y)$ is a regular´ring (see
Remark \ref{facets}(c)). Its localization $K[M']_P$ is
therefore also regular.

Now we turn to case (ii). We write the subfacet $G$ of
$\cone(M)$ corresponding to $Q$ as the intersection of facets
$F'$ and $F''$. Let $Q'$ and $Q''$ be the corresponding height
$1$ prime ideals. Since $X^x$ and $X^y$ cannot occur together
in $Q'$ or $Q''$, one of them, say $X^x$, lies in $Q'$ and
$X^y$ lies in $Q''$.  Since $M[-(F'\cap F'')]\iso
\ZZ^{d-2}\dirsum \ZZ_+^2$, the localization $K[M]_Q$ is a
regular local ring. Choosing bases in the summands, we write
$K[M[-(F'\cap F'')]]= K[Z_1^{\pm1},\dots,Z_{d-2}^{\pm1},U,V]$.
In this notation
$$
X^x-X^y=\mu U^{\hht_{F'}(x)}-\nu V^{\hht_{F''}(y)},\qquad
\mu,\nu\text{ monomials in } K[Z_1^{\pm1},\dots,Z_{d-2}^{\pm1}].
$$
The full localization $K[M]_Q$ is reached if we invert all
elements in $K[Z_1^{\pm1},\dots,Z_{d-2}^{\pm1},U,V]$ outside
the prime ideal generated by $U$ and $V$. The residue class
ring modulo $X^x-X^y$ is regular if (and only if) $X^x-X^y\in
Q_Q\setminus (Q_Q)^2$, and this is equivalent to $\hht_F(x)\le
1$ or $\hht_G(y)\le 1$.

For the converse implication (1) $\implies$ (2)one has to
reverse the arguments just used in the case (ii). First, the
regularity of $K[M']_P=K[M]_Q/(X^x-X^y)$ implies the regularity
of $K[M]_Q$ since the Krull dimension goes up by $1$ and the
number of generators of the maximal ideal by at most $1$. Now
Lemma \ref{regular} gives the structure of $M[-G]$. Moreover,
as said already, $K[M']_P=K[M]_Q/(X^x-X^y)$ is regular only if
$X^x-X^y\in Q_Q\setminus (Q_Q)^2$.
\end{proof}

We draw consequences similar to those of Theorem \ref{prime}.

\begin{corollary}\label{norm_mult}
Under the hypotheses of Corollary \ref{prime_mult} the
following are equivalent:
\begin{enumerate}
\item $K[M]/(X^{x_1}-X^{x_2},\dots,X^{x_{n-1}}-X^{x_n})$ is
    a normal domain;
\item for each face $F$ such that $\rank M-\dim F=n$ and
    $x_1,\dots,x_n\notin F$, one has the following:
    \begin{enumerate}
     \item $M[-F]\iso \ZZ^{d-n}\dirsum \ZZ_+^n$;
     \item at least $n-1$ of the $n$ nonzero numbers
         $\hht_{F_i}(x_j)$ are equal to $1$  for the
         facets $F_1,\dots,F_n$ containing $F$ and
         $j=1,\dots,n$.
     \end{enumerate}
\end{enumerate}
In particular, it is sufficient for (1) that all $n$ nonzero
heights $\hht_{F_i}(x_j)$ are equal to $1$ in the situation of
(2).
\end{corollary}

\begin{proof}
The equivalence of (1) and (2) follows by arguments entirely
analogous to those proving the theorem, except that the
critical localizations are now of type $M[-F]=\ZZ^{d-n}\dirsum
\ZZ_+^n$, and the regularity of the residue class ring modulo
$(X^{x_1}-X^{x_2},\dots,X^{x_{n-1}}-X^{x_n})$ is the crucial
condition.

For the last statement we observe that the normal monoid
$M[-G]$ splits into a direct sum of its unit group and a
positive (normal) affine monoid of rank $n$ \cite[2.26]{BG}.
The positive component must be isomorphic to $\ZZ_+^n$. In
fact, the standard map \cite[p. 59]{BG} sends it surjectively
and therefore isomorphically onto $\ZZ_+^n$.
\end{proof}

\begin{corollary}
Let $L$,$M$ and $N$ be normal affine monoids, $\phi:M\dirsum
N\to L$ a surjective homomorphism with $\rank L=\rank M+\rank
N-1$, and suppose that $\phi(x)=\phi(y)$ for $x\in M$, $y\in
N$, $x\neq 0$ or $y\neq 0$. Then the following are equivalent:
\begin{enumerate}
\item $K[L]=K[M\dirsum N]/(X^x-X^y)$ and $K[L]$ is normal,
\item $\hht_F(x)\le 1$ for all facets $F$ of $\cone(M)$ or
    $\hht_G(y)\le 1$ for all facets $G$ of $\cone(N)$.
\end{enumerate}
\end{corollary}

\begin{proof}
If (2) is satisfied, then $x-y$ is unimodular in $\gp(M\dirsum
N)$, and we need no longer think about the isomorphism
$K[L]=K[M\dirsum N]/(X^x-X^y)$.

In checking the equivalence of (1) and (2) in regard to
normality, one notes that the critical subfacets of
$\cone(M\dirsum N)$ are exactly the intersections $F'\cap F''$
where $F'$ is the extension of a facet of $\cone(M)$ not
containing $x$ and $F''$ extends a facet of $\cone(N)$ not
containing $y$, and \emph{all} such pairs $(F',F'')$ must be
considered.
\end{proof}

We want to state consequences for Ehrhart series similar to
Corollaries \ref{freesum} and \ref{rational_HH}. In the
situation of the free sum (and similarly in that analogous to
Corollary \ref{rational_HH}) one always has a homomorphism
$\phi:\cE(P)\dirsum \cE(Q)\to \cE(R)$ where $R=\conv(P\cup Q)$.
Set $L=\Im \phi$. By Corollary \ref{dirsum} we have
$\HH_L=(1-T_{m+1})\EE_P\EE_Q$. But $L$ and $\cE(R)$ generate
the same cone in $\RR^{m+1}$ (since $R=\conv(P\cap Q)$) and the
same subgroup of $\ZZ^{m+1}$ (since $(\ZZ^m\cap \RR
R)=(\ZZ^m\cap\RR P)+(\ZZ^m\cap\RR Q)$), and $\cE(R)$ is normal.
Therefore $\cE(R)$ is the normalization of $L$, and the
following statements are equivalent: (i) $L$ is normal, (ii)
$L=\cE(R)$, and (iii) $\HH_L=\EE_R$.

After these preparations we obtain \cite[Theorem 1.3]{Be3}. It
generalizes \cite[Corollary 1]{Bra} properly (see \cite[Remark
3.5]{Be3}).

\begin{corollary}\label{freesum_normal}
Let $R\subset \RR^m$ be a rational polytope that is the free
sum of the rational polytopes $P$ and $Q$, both containing $0$.
Then the following are equivalent:
\begin{enumerate}
\item At least in one of $P$ or $Q$ the origin has height
    $\le 1$ over all facets;
\item $\EE_R=(1-T_{m+1})\EE_P\EE_Q$.
\end{enumerate}
\end{corollary}

In the same way, as Corollary \ref{rational_HH} generalizes
Corollary \ref{freesum}, we can generalize Corollary
\ref{freesum_normal} and thus generalize \cite[Corollary
 5.8]{Be3}, but we must also generalize the condition $(\ZZ^m\cap
\RR R)=(\ZZ^m\cap\RR P)\dirsum (\ZZ^m\cap\RR Q)$. To this end
we say that a subset $A$ of $\ZZ^m$ is the \emph{$\ZZ$-affine
hull} of $B\subset \ZZ^m$ if
$$
A=\bigl\{a_1x_1+\dots+a_nx_n: n\ge 1, x_1,\dots,x_n\in B,\ a_1,\dots,a_n\in \ZZ,\
a_1+\dots+a_n=1\bigr\}.
$$
Note that the $\ZZ$-affine hull is the subgroup generated by
$B$ if $0\in B$.

\begin{corollary}\label{rational_EE}
Let $P,Q\subset \RR^m$ be rational polytopes such that
$\aff(P)$ and $\aff(Q)$ meet in a single point $p_0\in P\cap
Q$. Set $R=\conv(P\cup Q)$ and suppose that $\aff(R)\cap\ZZ^m$
is the $\ZZ$-affine hull of $(\aff(P)\cup\aff(Q))\cap\ZZ^m$.
Furthermore let $k$ be the smallest positive integer such that
$kp_0\in \ZZ^m$. Then the following are equivalent:
\begin{enumerate}
\item At least in one of $\cE(P)$ or $\cE(Q)$ the point $(kp_0,k)$
    has height $\le 1$ over all facets;
\item $\EE_R=(1-T^{(kp_0,k)})\EE_P\EE_Q$.
\end{enumerate}
\end{corollary}

Finally we derive \cite[Theorem 3]{BR} without using arguments
on triangulations.

\begin{corollary}
Let $M$ be an affine monoid such that $K[M]$ is Gorenstein and
let $X^w$, $w\in M$, generate the canonical module of $K[M]$.
Furthermore let $x_1,\dots,x_n\in M$ noninvertible elements
such that $w=x_1+\dots+x_n$. Then
$K[M]/(X^{x_1}-X^{x_2},\dots,X^{x_{n-1}}-X^{x_n})$ is again a
Gorenstein normal affine monoid domain and has dimension $\rank
M-(n-1)$.
\end{corollary}

\begin{proof}
The point $w$ is distinguished by the fact that it has height
$1$ over each facet. Therefore ``height vectors'' defined
$x_1,\dots,x_n$ are $0$-$1$-vectors with disjoint supports, and
Corollary \ref{norm_mult} applies. It yields that the residue
class ring is a normal affine monoid domain, and the Gorenstein
property is preserved modulo regular sequences.
\end{proof}

\end{document}